\documentclass{elsarticle}
\usepackage{amsmath}
\usepackage{amssymb}
\usepackage{amsfonts}
\usepackage{rotating}
\usepackage{graphicx}
\usepackage{floatflt,epsfig}
\usepackage{lineno,hyperref}
\usepackage{enumerate}
\usepackage{colortbl}
\usepackage{array,tabularx,tabulary,booktabs}
\usepackage{longtable}
\usepackage{multirow}
\usepackage{wrapfig}
\usepackage{subcaption}
\newcolumntype{^}{>{\currentrowstyle}}

\journal{ }
\newtheorem{definition}{Definition}

\newtheorem{theorem}{Theorem}

\newtheorem{question}{Question}

\begin{document}
\renewcommand{\abstractname}{Abstract}
\renewcommand{\refname}{References}
\renewcommand{\tablename}{Table.}
\renewcommand{\arraystretch}{0.9}
\sloppy

\begin{frontmatter}
\title{A new construction of strongly regular graphs with parameters of the complement symplectic graph}
\author[01]{Vladislav~V.~Kabanov}
\ead{vvk@imm.uran.ru}

\address[01]{Krasovskii Institute of Mathematics and Mechanics, S. Kovalevskaja st. 16, Yekaterinburg, 620990, Russia}

\begin{abstract}
The symplectic graph $Sp(2d, q)$ is the collinearity graph of the symplectic space of dimension $2d$ over a finite field of order~$q$. 
A $k$-regular graph on $v$ vertices is a divisible design graph  with parameters $(v,k,\lambda_1 ,\lambda_2 ,m,n)$ if its vertex set can be partitioned into $m$ classes of size $n$, such that any two different vertices from the same class have  $\lambda_1$ common neighbours, and any two vertices from different classes have  $\lambda_2$ common neighbours whenever it is not complete or edgeless.  
In this paper we propose a new construction of strongly regular graphs with the parameters of the complement of the symplectic graph using divisible design graphs. 
\end{abstract}

\begin{keyword} Strongly regular graph\sep  Divisible design graph

\vspace{\baselineskip}
\MSC[2020] 05C50\sep 05E30
\end{keyword}
\end{frontmatter}

\section{Introduction}

\begin{definition}
The {\em symplectic graph} $Sp(2d, q)$ is the collinearity graph of the symplectic space  of dimension $2d$ over a finite field of order~$q$. 
\end{definition}

For $d=1$ the symplectic graph is a $(q+1)$-coclique. For $d\geq 2$ this graph is strongly regular, with parameters
$$\left(\frac{q^{2d}-1}{q-1},\, \frac{q(q^{2d-2}-1)}{q-1},\, \frac{q^2(q^{2d-4}-1)}{q-1}+q-1,\, \frac{q^{2d-2}-1}{q-1}\right).$$ See, for example, \cite[Chapter 2]{BW} for definitions and properties of finite polar spaces and their collinearity graphs.

In 2016, A.~Abiad and W.H.~Haemers~\cite{AH} used Godsil-McKay switching to obtain strongly regular graphs with the same parameters as $Sp(2d,2)$ for all $d\geq 3$. In 2017, F.~Ihringer~\cite{FI} provided a new general construction of strongly regular graphs from the collinearity graph of a polar spaces of rank at least $3$ over a finite field of order~$q$. 
Recently A.E.~Brouwer, F.~Ihringer and W.M.~Kantor~\cite{BIK} described a switching operation on the collinearity graphs of polar spaces to obtain new strongly regular graphs which satisfy the so-called 4-vertex condition if the original graph comes from a symplectic polar space.

The article is organized as follows.
One type of divisible design graphs from~\cite{VK} is presented in Section~\ref{S1}.
In Section~\ref{S2} we discuss a regular decomposition of strongly regular graph.
We propose a new construction of strongly regular graphs with parameters
$$\left(\frac{q^{2d}-1}{q-1},\, q^{2d-1},\, q^{2d-2}(q-1),\, q^{2d-2}(q-1)\right)$$
 which are the parameters of the complement of~$Sp(2d, q)$ in Section~\ref{S3}. 
 Another construction from Hoffman coloring of strongly regular graphs is given in Section~\ref{S4}. 
 At the end, we discussed small examples, some open questions and the number of isomorphism classes of the graphs from our constructions.

\section{Construction of divisible design graphs}\label{S1}

\begin{definition} 
A $k$-regular graph on $v$ vertices is a {\em divisible design graph} with parameters $(v,k,\lambda_1 ,\lambda_2 ,m,n)$ if its vertex set can be partitioned into $m$ classes of size $n$, such that any two different vertices from the same class have  $\lambda_1$ common neighbours, and any two vertices from different classes have  $\lambda_2$ common neighbours  whenever it is not complete or edgeless. 
\end{definition}
The partition of a divisible design graph into  classes is called  a {\em canonical partition}.

Divisible design graphs were first introduced by W.H.~Haemers, H.~Kharaghani and M.~Meulenberg in~\cite{HKM}. 

W.D.~Wallis~\cite{WW}, D.G.~Fon-Der-Flaass~\cite{FF}, and  M.~Muzychuk~\cite{MM} proposed  a new construction of strongly regular graphs based on affine designs.
Similar idea is used in Construction~DDG  to obtain divisible design graphs.

Any $d$-dimensional affine space over a finite field of order $q$ is a point-hyperplane affine design with $q^d$ points. Any block contains $q^{d-1}$ points,  $q^{d-2}$ points are in the intersection of any two different blocks, and the number of blocks containing any two different points is $(q^{d-1}-1)/(q-1)$. 

\begin{definition}
A set $\mathcal{Q}$ equipped with a binary operation $\circ$ is called a {\em left quasigroup} if for all elements $i$ and $j$ in 
$\mathcal{Q}$ there is a unique element $h$ such that $i \circ h = j$. 
In other words, the mapping $h\mapsto i \circ h$ is a bijection of $\mathcal{Q}$ for any $i\in \mathcal{Q}$. 
\end{definition}
 
This section presents a construction of divisible design graphs. This construction first appeared in~\cite[Construction 1]{VK}. For the construction we use a left quasigroup and affine designs.
\smallskip

Let $\mathcal{D}_1, \dots ,\mathcal{D}_m$ be arbitrary affine designs all with parameters $(q,q^{d-2})$, where $m=(q^d - 1)/(q-1)$ is the number of parallel classes of blocks in each $\mathcal{D}_i$. For all  $i\in [m]:=\{1,2,\dots , m\}$, let $\mathcal{D}_i=(\mathcal{P}_i, \mathcal{B}_i)$. Parallel classes in each $\mathcal{D}_i$ are enumerated by integers from $[m]$ and
 $j$-th parallel class of $\mathcal{D}_i$ is denoted  by $\mathcal{B}_i^j$.
For any $x\in \mathcal{P}_i$, the block in the parallel class $\mathcal{B}_i^j$ which contains $x$ is denoted by $B_i^j(x)$.
\smallskip

Let $\mathcal{Q}$ be a left quasigroup on $[m]$ equipped with a binary operation $\circ$.
\smallskip

For every pair $i, j$ choose an arbitrary bijection 
$$\sigma_{i,j} : \mathcal{B}_i^{i\circ j} \rightarrow \mathcal{B}_j^{j\circ i}.$$ 

We require that $\sigma_{i,j}=\sigma_{j,i}^{-1}$ for $i\neq j$ and $\sigma_{i,i}$ is identity for all $i,j\in [m]$.
\medskip

{\bf Construction DDG.}

Let $\Gamma$ be a graph defined as follows:
\begin{itemize}
    \item The vertex set of $\Gamma$ is  $\displaystyle V=\bigcup_{i=1}^{m} \mathcal{P}_i.$ 
    \item Two different vertices $x\in \mathcal{P}_i$ and $y\in \mathcal{P}_j$ are adjacent in $\Gamma$ if and only if $$y \notin \sigma_{ij}(B_i^{i\circ j}(x))\quad \mathrm{for\, all} \quad i,j\in [m].$$ 
\end{itemize}
\medskip

\begin{theorem}\label{Th1}  
If $\Gamma$ is a graph from Construction~DDG, then  $\Gamma$ is a divisible design graph with parameters  
$$v = q^d (q^d - 1)/(q-1),\quad k = q^{d-1}(q^d - 1),$$
$$\lambda_1 = q^{d-1}(q^d - q^{d-1} - 1),\quad \lambda_2 = q^{d-2}(q-1)(q^d - 1),$$ $$m = (q^d - 1)/(q-1),\quad n = q^d.$$
\end{theorem}

Other known examples of  affine designs are Hadamard $3$-designs with~$q = 2$. Due to the peculiarities of Construction~DDG, such designs cannot be used in it.

\section{Regular decomposition of strongly regular graphs}\label{S2}
 
Let the set of vertices of a regular graph $\Gamma$ admit a partition $V(\Gamma)= X_1\cup X_2$ such that $\Gamma_1$ on $X_1$ and $\Gamma_2$ on $X_2$ are the induced subgraphs of $\Gamma$. This
decomposition is called  regular if $\Gamma_1$ and $\Gamma_2$ are regular.
There is the incident structure $\mathcal{D}$ having block set  $X_1$ and point set $X_2$, and incidence given by adjacency in $\Gamma$. 

If $\Gamma$ is strongly regular and $\Gamma_1$ is regular, then some inequalities for the eigenvalues of $\Gamma$ was found by W.H.~Haemers and D.G.~Higman in~\cite[Theorem 2.2]{HH}. We need a special case of this theorem, noted by E.R.~van~Dam in \cite[Section~4.5.1]{vD}. 

Let $\Gamma$ be a primitive strongly
regular graph on $v$ vertices, with spectrum $\{k^1, r^f, s^g\}$. If $\Gamma_1$ on $X_1$ is a regular graph and $\Gamma_2$ on $X_2$ is a coclique  of size $vs/(s-k)$ (known as a Hoffman coclique), then the induced subgraph $V(\Gamma)\setminus C$ is a regular, connected graph with spectrum
$$\{(k + s)^1, r^{f-c+1}, (r + s)^{c-1}, s^{g-c}\}.$$
Moreover, if $c < g$, then the induced subgraph $V(\Gamma)\setminus C$ has four distinct eigenvalues. 

The spectrum of any divisible design graph with parameters $(v,k,\lambda_1,\lambda_2,m,n)$ can be calculated using its parameters as follows:
$$\{k^1,\sqrt{k-\lambda_1}^{f_1},-\sqrt{k-\lambda_1}^{f_2},
\sqrt{k^2-\lambda_2 v}^{g_1},-\sqrt{k^2-\lambda_2 v}^{g_2}\}.$$
Moreover, $f_1+f_2 =  m(n - 1)$ and $g_1+g_2 = m-1$ \cite[Lemma 2.1]{HKM}.

Therefore, the spectrum of any divisible design graph from Theorem~\ref{Th1} have four distinct eigenvalues  
$$\{q^d(q^{d-1} - 1),\,  q^{d-1},\,  0,\,  -q^{d-1}\},$$ with multiplicities 
$$\left\{1, \cfrac{(q^d -1)(q^d -q)}{2(q-1)}, \cfrac{q^d -q}{q-1}, \cfrac{(q^d -1)(q^d +q -2)}{2(q-1)}\right\}.$$

Note that the eigenvalue $0$ multiplicity is $(q^d -q)/(q-1)$.
This observation allows us to use any divisible design graph from Construction~DDG as $\Gamma_1$ and a coclique of size $c=(q^d -q)/(q-1)+1=(q^d -1)/(q-1)$ as $\Gamma_2$  to construct a new strongly regular graph.

\section{Construction of strongly regular graphs}\label{S3}

In this section, a construction of strongly regular graphs with
parameters $$\left((q^{2d}-1)/(q-1), q^{2d-1}, q^{2d-2}(q-1), q^{2d-2}(q-1)\right)$$ 
is given. These parameters are known as the parameters of the complement of $Sp(2d,q)$. 
\smallskip

Let $\Gamma^\ast$ be a divisible design graph with parameters $$v^\ast = q^d (q^d - 1)/(q-1),\quad k^\ast = q^{d-1}(q^d - 1),$$
$$\lambda_1^\ast = q^{d-1}(q^d - q^{d-1} - 1),\quad \lambda_2^\ast = q^{d-2}(q-1)(q^d - 1)$$ 
 on the vertex set $V^\ast$.
The canonical partition of $\Gamma^\ast$ consists of $m^\ast=(q^d - 1)/(q-1)$ classes which have the size $q^d$.
Let  $\mathcal{D}^\ast=(\mathcal{P}^\ast,\mathcal{B}^\ast)$ be a symmetric 
$2$-$((q^{d} -1)/(q-1), q^{d-1}, q^{d-2}(q-1))$ design. 
\smallskip

Let $\phi$ be an arbitrary bijection from the set of all canonical classes of $\Gamma^\ast$ to the set of blocks $\mathcal{B}^\ast$.
\medskip

{\bf Construction SRG1.}

Let $\Gamma$ be a graph defined as follows:
\begin{itemize}
    \item The vertex set of $\Gamma$ is $\displaystyle V = V^\ast\cup \mathcal{P}^\ast.$ 
   \item Two different vertices from $V^\ast$  are adjacent in $\Gamma$ if and only if they are adjacent in $\Gamma^\ast$.
The set $\mathcal{P}^\ast$ is a coclique in $\Gamma$. 
 A vertex $x$ in $V^\ast$ from any canonical class $\mathcal{P}_i$ is adjacent to a vertex $y$ in $\mathcal{P}^\ast$ if and only if $y$ belongs to the block $\phi(\mathcal{P}_i)$, where $i\in [m^\ast]$. 
   \end{itemize}
\medskip

\begin{theorem}\label{Th2}
If $\Gamma$ is a graph from Construction~SRG1, then  $\Gamma$ is a strongly regular graph with parameters  
$$((q^{2d}-1)/(q-1), q^{2d-1}, q^{2d-2}(q-1), q^{2d-2}(q-1)).$$ 
\end{theorem}
\begin{proof}
Let $\Gamma$ be a graph from Construction~SRG1. The number of vertices in $\Gamma$ is equal to $$q^d (q^d -1)/(q-1) + (q^d -1)/(q-1) = (q^{2d}-1)/(q-1).$$
If $x$ is a vertex of $\Gamma$ from $\mathcal{P}_i$, where $i\in [m^\ast]$, then there are 
$$k^\ast + |\phi(\mathcal{P}_i)|=q^{d-1}(q^d - 1) + q^{d-1} = q^{2d-1}$$ vertices in $\Gamma(x)$. 
If $y$ is a vertex of $\Gamma$ from $\mathcal{P^\ast}$, then
$\Gamma(y)$ consists of $q^{d-1}$ canonical classes of size $q^d$ from $V \setminus \mathcal{P^\ast}$. 
Hence, $\Gamma$ is a regular graph of degree~$q^{2d-1}.$

Let $x$ and $y$ be two different vertices from $\Gamma$ and $\lambda(x,y)$ be the number of their common neighbours in $\Gamma$. This number depends only on the relative position of $x$ and $y$ in $\Gamma^\ast$ and $V^\ast$. Consider all the possibilities of placing $x$ and $y$ in $\Gamma$.
\begin{itemize}
    \item[(1)] If  $x$ and $y$ belong to the same canonical class of $\Gamma^\ast$,
 then $x$ and $y$ have $\lambda_1^\ast =q^{d-1}(q^d - q^{d-1} - 1)$ common neighbours in $V^\ast$ and $q^{d-1}$ common neighbours in~$\mathcal{P}^\ast$. 
 Thus, $$\lambda(x,y)=q^{d-1}(q^d - q^{d-1} - 1)+q^{d-1}=q^{2d-2}(q-1).$$ 
    \item[(2)]  If  $x$ and $y$ belong to different canonical classes of $\Gamma^\ast$, then $x$ and $y$ have $\lambda_2^\ast =q^{d-2}(q-1)(q^d - 1)$ common neighbours in $V^\ast$ and $q^{d-2}(q-1)$ common neighbours in~$\mathcal{P}^\ast$. 
Thus, $$\lambda(x,y)=q^{d-2}(q-1)(q^d - 1)+q^{d-2}(q-1)=q^{2d-2}(q-1).$$ 
    \item[(3)] If  $x$ and $y$ belong to  $\mathcal{P}^\ast$,  then $x$ and $y$ have  $q^{d-2}(q-1)$  times $q^d$ common neighbours in $V^\ast$, and they have no  common neighbours in~$\mathcal{P}^\ast$. 
Thus, $$\lambda(x,y)=q^d q^{d-2}(q-1)=q^{2d-2}(q-1).$$ 
    \item[(4)]  If $x\in V^\ast$ and $y\in \mathcal{P}^\ast$, then  $x$ and $y$ have $q^d - q^{d-1}$ multiplied by $q^{d-1}$ common neighbours in $V^\ast$ and have no  common neighbours in~$\mathcal{P}^\ast$. 
Thus, $$\lambda(x,y)=q^{d-1}q^d - q^{d-1}=q^{2d-2}(q-1).$$ 
\end{itemize}

Therefore, in all cases the number of common neighbours of $x$ and $y$ in $\Gamma$ equals 
$q^{2d-2}(q-1)$. Hence, $\Gamma$ is a strongly regular graph with parameters 
$$((q^{2d}-1)/(q-1), q^{2d-1}, q^{2d-2}(q-1), q^{2d-2}(q-1)).$$ This completes the proof of Theorem~\ref{Th2}.
\end{proof} 

\section{Construction from Hoffman coloring}\label{S4} 

Let $\Gamma$ be a strongly regular graph of valency $k$ and smallest eigenvalue $s$. By P.~Delsarte~\cite{D} clique in $\Gamma$ of size $1 - k/s$ is called a {\it Delsarte clique}. 

A natural question arises: is it possible to construct a strongly regular graph using a Delsarte clique instead of a coclique on $\mathcal{P}^\ast$ in Construction~SRG1. The answer is yes, but we must use a different divisible design graph. 

A.J.~Hoffman~\cite{AJH} proved that the chromatic number of any graph with largest eigenvalue $k$ and
smallest eigenvalue $s$ is at least $vs/(s-k)$. The coloring of a graph meeting this bound is called a Hoffman coloring. D.~Panasenko and L.~Shalaginov~\cite{PS} found divisible design graphs using Hoffman coloring of strongly regular graphs with parameters $(v, k, \mu + 2, \mu)$. 

Suppose $\Delta$ is a strongly regular graph with parameters $(v, k, \mu + 2, \mu)$ and has a Hoffman coloring with Hoffman cocliques of size $n$. Let $A$ be the adjacency matrix of $\Delta$.
Let $K = K_{(m, n)}$ and $I = I_v$. By~\cite[Construction 16]{PS} $A + K - I$ is the adjacency matrix of a divisible design graph with parameters $$(mn, k + n - 1, n + \mu - 2, 
2k/(m-1) + \mu, m, n).$$ 

Let $\Gamma^\ast$ be a divisible design graph with parameters $(v^\ast, k^\ast, \lambda_1^\ast, \lambda_2^\ast, m^\ast,n^\ast)$ on the vertex set $V^\ast$  from \cite[Construction 16]{PS}.
Let  $\mathcal{D}^\ast=(\mathcal{P}^\ast,\mathcal{B}^\ast)$ be a symmetric $2$-$(v^\infty, k^\infty, \lambda^\infty)$ design. 
\smallskip

Let $v^\infty =m^\ast$ and $\phi$ be an arbitrary bijection from the set of canonical classes of $\Gamma^\ast$ to the set of blocks $\mathcal{B}^\ast$.
\medskip

{\bf Construction SRG2.}
Let $\Gamma$ be a graph defined as follows:
\begin{itemize}
    \item The vertex set of $\Gamma$ is $\displaystyle V = V^\ast\cup \mathcal{P}^\ast.$ 
   \item Two different vertices from $V^\ast$  are adjacent in $\Gamma$ if and only if they are adjacent in $\Gamma^\ast$.\medskip
The set $\mathcal{P}^\ast$ is a clique in $\Gamma$. 
 A vertex $x$ in $V^\ast$ from the canonical class $\mathcal{P}_i$, where $i\in [m^\ast]$, is adjacent to a vertex $y$ in $\mathcal{P}^\ast$ if and only if $y$ belongs to the block $\phi(\mathcal{P}_i)$. 
   \end{itemize}
\medskip

If the following equalities hold $$\sqrt{k - \mu + 1} + n + \mu - 2 = m-2+n\lambda^\infty = 2k/(m-1) + \mu +  \lambda^\infty,$$ then from Construction~SRG2 we have a strongly regular graph with parameters $$(m(n+1), k + \sqrt{k - \mu + 1} + n - 1, \sqrt{k - \mu + 1} + n + \mu - 2, \sqrt{k - \mu + 1} + n + \mu - 2).$$

\section{Small examples}\label{S5}

\subsection{Examples from Construction~SRG1}

If $q=2$ and $d=2$, then, by Theorem~\ref{Th1}, a divisible design graph has parameters $(12,6,2,3,3,4)$.
The line graph of the octahedron is a unique graph with such parameters.
By Theorem~\ref{Th2}, we have the triangular graph $T(6)$ which is a strongly regular graph with parameters $(15,8,4,4)$.

 D.~Panasenko and L.~Shalaginov~\cite{PS} found all divisible design graphs up to 39 vertices by direct computer calculations, except for the three tuples of parameters: $(32,15,6,7,4,8)$, $(32,17,8,9,4,8)$, $(36,24,15,16,4,9)$. These cases proved to be very difficult to handle.
 
If $q=3$ and $d=2$, then, by Theorem~\ref{Th1}, we have parameters $(36,24,15,16,4,9)$. 
By Theorem~\ref{Th2}, from any divisible design graph with parameters $(36,24,15,16,4,9)$ and a $2$-$(4,3,2)$ design we can construct a strongly regular graph with parameters $(40,27,18,18)$. The complement of this graph has parameters $(40,12,2,4)$. 
All strongly regular  graphs with parameters $(40,12,2,4)$ were found by E.~Spence~\cite{ES}. 
There are exactly $28$ non-isomorphic strongly regular graphs with parameters $(40,12,2,4)$.
Only the first one of them does not have $4$-cliques. If strongly regular graph $\Gamma$
 with parameters $(40,12,2,4)$ has a regular $4$-clique $C$, then the induced subgraph on $V(\Gamma)\setminus C$ in the complement of  $\Gamma$ is a divisible design graph  with parameters $(36,24,15,16,4,9)$. It turns out that there are $87$  non-isomorphic divisible design graphs  with such parameters. The adjacency matrices of all these graphs were calculated by D.~Panasenko by checking the graphs found by E.~Spence~\cite{ES}.Them are available on the web page \url{ http://alg.imm.uran.ru/dezagraphs/ddgtab.html}.

\subsection{Examples from Construction~SRG2}

There are 56 divisible design graphs with parameters $(28, 15, 6, 8, 7, 4)$ and spectrum $\{15^1, 3^7, 1^6, -3^{14}\}$ from either the Triangular graph $T(8)$ or one of the Chang graphs \cite[Construction 16]{PS}.  Construction~SRG2 gives us strongly regular graphs with parameters $(35,18,9,9)$. Remark that strongly regular graphs with parameters $(35,18,9,9)$ were complete enumerated by B.D.~McKay and E.~Spence in~\cite{MKS}. There are 3854 strongly regular graphs with parameters $(35,18,9,9)$. 499 of them can be obtained from Construction~SRG2.

\section{The number of isomorphism
classes of graphs from Construction~DDG and Construction~SRG1}\label{S6} 

In this section we discuss
how many non-isomorphic divisible design graphs and strongly regular graphs we can get by Construction~DDG and Construction~SRG1, respectively, for given $q$ and $d$?

Using the same arguments as M.~Muzychuk in \cite[Proposition 3.5]{MM} we obtain a lower bound for the number of non-isomorphic
divisible design graphs from Construction~DDG as follows $$\frac{(q!)^m}{(q^d m^2)^{q^d m}(q^{d+1} m)^{m-1}},$$ where $m=(q^d -1)/(q-1)$.
Let $D_1$ be the number of non-isomorphic symmetric  $2$-$\left(\frac{q^d -1}{q-1},\frac{q^{d-1} -1}{q-1},\frac{q^{d-2} -1}{q-1}\right)$ designs. 
If $\Gamma^\ast$ is a divisible design graph with parameters from Construction~DDG, then there are at least $D_1$ pairwise non-isomorphic strongly regular graphs, by Construction~SRG1, with $\Gamma^\ast$ as an induced subgraph. 
To obtain the number of isomorphism classes of strongly regular graphs from Construction~SRG1, we need to estimate the number of Delsarte cliques in a graph with the symplectic graph  parameters.
For $q\gg 1$, $d\gg 1$ this number is at most $D_2=q^{2d-1}(q^{d-2})^{q^{d-2}}$.
Hence, the number of isomorphism classes of strongly regular graphs from Construction~SRG1 is at least
$$\frac{D_1(q! )^m}{D_2(q^d m^2)^{q^d m}(q^{d+1} m)^{m-1}}.$$

\section{Open questions}\label{S7}
\begin{question}
 Is it possible to obtain all non-isomorphic divisible design graphs with parameters  
$$v = q^d (q^d - 1)/(q-1), k = q^{d-1}(q^d - 1),$$
$$\lambda_1 = q^{d-1}(q^d - q^{d-1} - 1), \lambda_2 = q^{d-2}(q-1)(q^d - 1),$$ $$m = (q^d - 1)/(q-1),\quad n = q^d$$ from Construction~DDG for given $q$ and $d$?
\end{question}

A similar question for strongly regular graphs from Construction~SRG1 has a “no” answer. There is one strongly regular graph with parameters $(40, 12, 2, 4)$ which we cannot  obtain by Construction~SRG1.

\begin{question}
Is it true that among all graphs
with the same parameters as the symplectic graph, the symplectic graph contains the maximum possible number of Delsarte cliques?
\end{question}
If $d=2$, then the answer is "yes"  since the symplectic graph $Sp(4,3)$ is one of two generalized quadrangles  $GQ(3,3)$.

\subsection*{Acknowledgements}

The author is very grateful to Alexander~Gavrilyuk and Misha~Muzychuk for  their valuable remarks which improve the paper.

\end{document}